\documentclass{article}
\usepackage[utf8]{inputenc}\usepackage{url}
\usepackage{amsmath,amsfonts}
\usepackage{parskip}
\usepackage[margin=1.4in]{geometry}
\usepackage{authblk}
\usepackage{dsfont}
\usepackage{mathtools}
\usepackage{amsthm}
\usepackage{multirow}
\usepackage{multicol}
\begingroup
    \makeatletter
    \@for\theoremstyle:=definition,remark,plain\do{%
        \expandafter\g@addto@macro\csname th@\theoremstyle\endcsname{%
            \addtolength\thm@preskip\parskip
            }%
        }
\endgroup
\DeclarePairedDelimiter{\ceil}{\lceil}{\rceil}
\DeclarePairedDelimiter{\floor}{\lfloor}{\rfloor}

\usepackage[colorinlistoftodos]{todonotes}

\newcommand{\ZZ}{\mathds{Z}}
\newcommand{\RR}{\mathds{R}}
\newcommand{\EE}{\mathds{E}}
\newcommand{\PP}{\mathds{P}}
\newcommand{\G}{\mathcal{G}}
\newcommand{\cN}{\mathcal{N}}
\newcommand{\NN}{\mathds{N}}
\newcommand{\e}{\varepsilon}
\newcommand{\Vol}{\mathrm{Vol}}
\newcommand{\cU}{\mathcal{U}}

\newcommand{\be}{\begin{equation}}
\newcommand{\ee}{\end{equation}}

\newtheorem{thm}{Theorem}[section]
\newtheorem{lemma}[thm]{Lemma}

\newtheorem{prop}[thm]{Proposition}

\newtheorem{Def}[thm]{Definition}

\newtheorem{?}[thm]{Problem}

\theoremstyle{remark} 
\newtheorem{remark}[]{Remark}

\title{Homotopy Types of Random Cubical Complexes}

\author[1]{K. Alex Dowling}
\author[2]{Erik Lundberg\thanks{The author's research is partly supported by the Simons Foundation, under the grant 712397.}}
\affil[1]{Rutgers University, kenneth.dowling@rutgers.edu}
\affil[2]{Florida Atlantic University, elundber@fau.edu}

\usepackage[pagewise]{lineno}

\begin{document}

\maketitle

\begin{abstract} 
We study the topology of
a random cubical complex associated to Bernoulli site percolation on a cubical grid. We begin by establishing a limit law for homotopy types.  
More precisely, 
looking within an expanding window,
we define a sequence of normalized counting measures 
(counting connected components according to homotopy type),
and we show that this sequence of random probability measures converges in probability to a deterministic probability measure.
We then investigate the dependence of the limiting homotopy measure on the coloring probability $p$,
and our results show a qualitative change in the homotopy measure as $p$ crosses the percolation threshold $p=p_c$.
Specializing to the case of $d=2$ dimensions, we also present empirical results
that raise further questions on the $p$-dependence of the limiting homotopy measure. 

\smallskip

\end{abstract}


\section{Introduction}

Imagine coloring each cell in a cubical grid
independently black with probability $p$ and white with
probability $(1-p)$.  Let $\Phi$ denote the interior of the union of the black cubes.
We are interested in the following problem.

\noindent {\bf Problem.}
\emph{Study the asymptotic behavior of the topology of $\Phi$ restricted to an increasingly large viewing window,
and investigate how the outcome depends on $p$.}

This problem fits
into the emerging field of \emph{stochastic topology} investigating the topology of random manifolds and random complexes (the set $\Phi$ can be viewed as a random cubical complex).
See the survey papers \cite{BobrowskiKahle}, \cite{KahleSurvey}.
Motivation for stochastic topology includes applications in areas such as 
topological data analysis \cite{BobrowskiKahlePers}, \cite{HiShTr},
manifold learning \cite{NSW},
complex systems \cite{DayKaliesMis},
and quantum chaos \cite{SarnakWigman}.

Percolation theory has been a major source of inspiration in this area, and the setup for many studies in stochastic topology can be viewed as attaching novel topological statistics to classically-studied percolation models.
For instance, the random clique complexes studied in \cite{KahleClique} are Erd\"os-Renyi random graphs with additional simplicial structure, the random geometric complexes studied in \cite{Kahle2011} coincide with classical models of continuum percolation \cite{MeesterRoy}, and the current paper follows \cite{HiSh2016}, \cite{WerWr2016}, \cite{HiTs} investigating topological aspects of the Bernoulli site percolation model \cite{Grimmett}.

Most of the previous work in stochastic
topology has focused on homology.
In \cite{KahleClique},
Kahle investigated asymptotics for
Betti numbers of clique complexes.
In \cite{Kahle2011}, \cite{BobrowskiMukherjee}, and \cite{YSA} the authors study asymptotics for Betti numbers of random geometric complexes
(simplicial complexes built over random point clouds with faces determined by pairwise distances).
In her thesis \cite{Erika}, Rold\'an investigated the homology of random polynominoes which are a form of random two-dimensional cubical complexes.
Hiraoka and Tsunoda \cite{HiTs} recently studied asymptotics for Betti numbers of random
cubical complexes associated to the site percolation process,
and they prove a law of large numbers and central limit theorem for Betti numbers.
They also showed that the asymptotic law for each Betti number depends continuously on the parameter $p$ (even across the so-called percolation threshold $p=p_c$).

The recent study \cite{ALL} instead focuses on homotopy, which represents a more refined level of information than homology. The authors of \cite{ALL} used methods from \cite{SarnakWigman}, \cite{NazarovSodin} to prove a limit law for homotopy types of random geometric complexes. The first part of the current paper adapts this approach to establish a limit law for the homotopy type of a random cubical complex.

Our setting is simpler than the random nodal sets of \cite{SarnakWigman} or the random geometric complexes of \cite{ALL}, but in addition to establishing a homotopy limit law, we are able to go further in addressing the second part of the above problem asking to ``investigate how the outcome depends on $p$''.
Namely, we study the tail decay of coefficients in the limiting homotopy measure, and our results show a qualitative change in the decay rate when $p$ crosses the so-called \emph{percolation threshold $p_c$};
the tail decay rate is exponential when $0<p<p_c$, and the tail decay rate is subexponential (slower than exponential) when $p_c<p<1$, see theorems \ref{thm:exptailmu} and \ref{thm:subexptailmu} below.


\subsection{A random cubical complex associated to site percolation}

Fix a dimension $d \geq 2$, and consider the lattice $\mathds{Z}^d\subset\mathds{R}^d$ of points with integer coordinates.  Color the lattice points independently at random; designate each point $x\in\mathds{Z}^d$ colored black with probability $p\in(0,1)$ and colored white with probability $1-p$. Call the random set of all black points $Z$.  We will want to be able to refer to the $d$-dimensional cube with lower corner $x\in\mathds{Z}^d$ and size $\ell$:
$$B_\ell(x)=\prod_{i=1}^d[x_i,x_i+\ell].$$ 

\noindent
From $Z$ we build a random
cubical complex
by taking the interior of the union of closed unit cubes:
$$\Phi=\bigg(\bigcup\limits_{z\in Z}B_1(z)\bigg)^{\mathrm{o}}.$$
We note that
the connected components of $\Phi$,
also referred to as \emph{clusters},
coincide with so-called \emph{polyominoes},
which are also referred to as \emph{lattice animals}.

Recall \cite[Ch. 2]{Munkres} that two topological spaces $X$ and $Y$ are said to be homotopy equivalent
if there are maps
$f:X \rightarrow Y$,
$g:Y \rightarrow X$
such that $g \circ f$ is homotopic to the identity map $\text{id}_X$, and $f \circ g$ is homotopic to the identity map $\text{id}_Y$.
While we will often work directly with the topological space $\Phi$ viewed as an open subset of $\RR^d$, it is sometimes convenient to use the homotopically equivalent cubical complex described in the remark below. We refer the reader to \cite[Ch. 2]{Mischaikow} for an overview of algebraic topology for cubical complexes.

\begin{remark}\label{rmk:cubicalcomplex}
The set $\Phi$ is homotopy equivalent
to the following cubical complex.
First shift the set $Z$ by $(1/2,1/2,...,1/2)$ to obtain the vertex set $F_0 = Z+(1/2,1/2,...,1/2)$ consisting of centers of all the unit cubes in $\Phi$.
Build an edge set $F_1$ by including every edge between a pair of adjacent (i.e., differing by only one step in only one coordinate) vertices in $F_0$, and for $k=2,..,d$ inductively build a set of $k$-dimensional faces by including a $k$-dimensional cube if all of its $(k-1)$-dimensional faces appear in $F_{k-1}$. To see that $\Phi$ is homotopy equivalent to the cubical complex $F:= \cup_{k=0}^d F_k$ we use the nerve lemma as follows.
For each face $f$ in $F_k$, let $U_f$ denote the the interior of the union of closed unit cubes centered at the vertices of $f$.
The collection $\cU = \{U_f\}_{f\in F}$ forms an open cover of $\Phi$, and, since the $U_f$ are convex, the nerve lemma \cite[Cor. 4G.3]{Hatcher} yields that the nerve $N(\cU)$ of $\cU$ is homotopy equivalent to $\Phi$. As for $F$, take $U^*_f = U_f \cap F$, and let $\cU^*$ denote the open cover $\{U^*_f\}_{f\in F}$ of $F$. 
Each of the sets $U^*_f$, while not necessarily convex, is contractible, which can be seen from a straight line homotopy from $U^*_f$ to the face $f$, which is contractible. Moreover, the intersection of any sets in $\cU^*$ is either empty or is again a set in $\cU^*$.
So the nerve lemma implies that $F$ is homotopy equivalent to the nerve $N(\cU^*)$. But the nerves $N(\cU)$ and $N(\cU^*)$ are isomorphic: any intersection of sets $U_f$ is nonempty if and only if the intersection of the corresponding $U^*_f$ is nonempty. It follows that $\Phi$ is homotopy equivalent to the cubical complex $F$ as claimed.
\end{remark}

\begin{figure}[ht]
    \centering
    \includegraphics[scale=0.6]{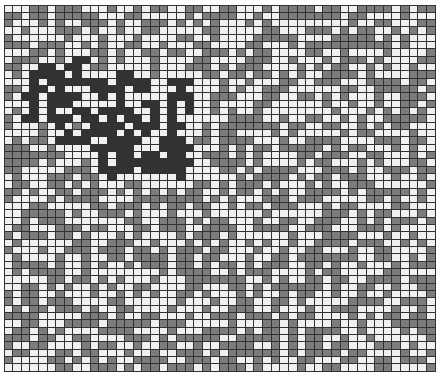}
    \caption{A random coloring in $d=2$ dimensions with coloring probability $p=0.5$.  One of the connected components of $\Phi$ is displayed in a darker shade.
    }
    \label{fig:perc50}
\end{figure}

\subsection{Distribution of homotopy types}

\noindent
Following \cite{SarnakWigman} and \cite{ALL} we will be interested in the connected components of the (random) set $\Phi$ in a window of expanding size. 
For an integer side length $L>0$, define the window of size $L$ centered at $x\in\mathds{Z}^d$ as:
$$W_L(x)=\prod_{i=1}^d\bigg[x_i- \floor[\bigg]{\frac{L}{2}}, x_i+ \ceil[\bigg]{\frac{L}{2}} \bigg],$$
which has volume $L^d$. 

\noindent
Let $\mathcal{G}$ denote the set of equivalence classes of finite connected $d$-dimensional cubical complexes up to homotopy equivalence. This set is countable (it is a countable union of finite sets). 

Let $\Phi_L$ denote the union of connected components of $\Phi$ that are completely contained within the window $W_L(0)$,
and let us next define the random homotopy measure
$\mu_L$ on $\G$ associated to $\Phi_L$:
\begin{equation}\label{eq:mudef}
\mu_L := \frac{1}{b_0(\Phi_L)} \sum \delta_{[s]} = \sum a_{L,\gamma} \delta_\gamma,
\end{equation}
where the first sum is over all connected components $s$ of $\Phi_L$, $b_0(\Phi_L)$ denotes the zeroth Betti number (number of connected components) of $\Phi_L$, and $[s]$ denotes the homotopy type of $s$ (i.e., the equivalence class of connected cubical complexes that are homotopy equivalent to $s$).
In the second sum appearing in (\ref{eq:mudef}),
we have collected terms
so the sum is over
homotopy types $\gamma \in \G$
and the coefficient $a_{L,\gamma}$ is the number of connected components in $\Phi_L$ with homotopy type $\gamma$ normalized by $b_0(\Phi_L)$.
Note that $\mu_L$ has total mass one (i.e., it is a probability measure).

\noindent

Before we state our limit law for the sequence of homotopy measures $\mu_L$,
we recall the notion of convergence in probability.
We say that a sequence $\{ \mu_n \}_{n \in \NN}$ of random measures \emph{converges in probability} to a measure $\mu$ if for every $\e>0$ we have $$\lim_{n \rightarrow \infty} \PP\{d(\mu_n, \mu)>\e\}=0,$$
where $$ d(\mu_1, \mu_2):=\sup_{A} \left|\mu_1(A)-\mu_2(A)\right|$$
denotes the \emph{total variation distance}.

\begin{thm}\label{thm:main}
The random measure $\mu_L$ converges in probability to a deterministic probability measure $\mu$ whose support is $\G$.
\end{thm}

We next consider how the limiting distribution of homotopy types, captured by the measure $\mu$, depends on the coloring probability $p$.

Let $C_0$ denote the (possibly empty) connected
component of $\Phi$ containing the interior of the cube $B_1(0)$.
Recall from percolation theory \cite{Grimmett} that the critical probability $0<p_c<1$ is
defined as 
\begin{equation}\label{eq:defpc}
p_c := \inf_p \{p: \theta(p) > 0 \},
\end{equation} where $\theta(p)$ denotes the probability that the size $|C_0|$ of the component $C_0$ is infinite.
It is known that the probability that
$\Phi$ has an infinite connected component is zero when $p<p_c$, and when $p > p_c$
the probability that $\Phi$ has an infinite connected component is one.  In the latter case, the infinite component is unique almost surely.

In the case of $d=2$ dimensions, it is simple to index the elements of $\G$ by the natural numbers (each type can be represented as a wedge of some number of circles),
but in $d \geq 3$ dimensions
explicit enumeration is much more complicated.
Thus, in order to formulate our results concerning the tail decay of $\mu$, we introduce \emph{tail sets} of homotopy types based on ``binning'' the types with respect to a specified Betti number.
To be precise, for each $k=1,2,...,d-1$
define the tail set $T_n = T_n(k,d)$ as
\begin{equation}\label{eq:tailset}
T_n(k,d) := \{ \gamma \in \G : \beta_k(\gamma) \geq n \},
\end{equation}
where $\beta_k$ denotes the $k$th Betti number of $\gamma$. We are interested in the decay rate of the measure $\mu(T_n)$ of the tail set as $n \rightarrow \infty$.

\begin{thm}[Exponential tail for $p<p_c$]\label{thm:exptailmu}
Fix $0 < k < d$, and suppose that $0<p<p_c$.
There exists $\alpha = \alpha(p,k,d)>0$ such that
\begin{equation}\label{eq:expdecay}
 \mu(T_n) \leq \exp\left\{ -\alpha \cdot n \right\}, \quad \text{ for all } n \geq 1,
\end{equation}
where $T_n = T_n(k,d) \subset \G$ denotes the tail set defined above in (\ref{eq:tailset}).
\end{thm}

\begin{thm}[Subexponential tail for $p>p_c$]\label{thm:subexptailmu}
Fix $0 < k < d$, and suppose that $p_c<p<1$.
There exists $\eta = \eta(p,k,d)>0$ such that
\begin{equation}\label{eq:subexpdecay}
\mu(T_n) \geq \exp \left\{ -\eta \cdot n^{(d-1)/d} \right\},
\end{equation}
for all $n$ sufficiently large,
where $T_n=T_n(k,d) \subset \G$ denotes
the tail set defined in (\ref{eq:tailset}).
\end{thm}

\subsection*{Outline of the paper}

We present the proof of Theorem \ref{thm:main} in Section \ref{sec:coeff}.
In Section \ref{sec:perc},
we investigate dependence of $\mu = \mu(p)$ on 
the coloring probability $p$.
We prove Theorem \ref{thm:exptailmu}
in Section \ref{sec:exptail}
and Theorem \ref{thm:subexptailmu} in Section \ref{sec:subexptail}, where we employ a novel adaptation of techniques from percolation theory combined with several topological considerations.
We prove continuity of the individual coefficients in the homotopy limit law in Section \ref{sec:continuity},
and in Section \ref{sec:empirical} we present empirical results leading to interesting questions and conjectures regarding the dependence of $\mu$ on $p$.
We conclude with some remarks and open problems in Section \ref{sec:concl}.

\subsection*{Acknowledgements}
The theme of investigating homotopy types with attention toward the percolation threshold was partly inspired by a question posed to the second author by Matthew Kahle.
We would also like to thank anonymous referees for helpful comments and corrections.


\section{The Homotopy Limit Law: Proof of Theorem \ref{thm:main}}\label{sec:coeff}

\begin{Def}\label{def:N}
(homotopy type counting functions) Let $Y_1,Y_2$ be topological spaces, and $Z\in\mathcal{G}$ a homotopy class. \begin{itemize}
    \item Let $\mathcal{N}(Y_1,Y_2,;Z)$ denote the number of connected components of $Y_1$ that are entirely contained in the interior of $Y_2$ and are members of the homotopy class $Z$.
    \item Let $\mathcal{N}^*(Y_1,Y_2,;Z)$ denote the number of connected components of $Y_1$ that have nonempty intersection with $Y_2$ and are members of $Z$.
\end{itemize}
\label{main}
\end{Def}


We first prove convergence for
each homotopy count normalized by the volume $L^d$ of $W_L(0)$, 
$$ c_{L,\gamma}=\frac{\mathcal{N}(\Phi,W_{L}(0);\gamma)}{L^d}.$$

\subsection{Convergence of $c_{L,\gamma}$}

\begin{prop}\label{prop:const} For every $\gamma\in\mathcal{G}$ there exists a constant $c_{\gamma}$ such that the random variable 
$$c_{L,\gamma}=\frac{\mathcal{N}(\Phi,W_{L}(0);\gamma)}{L^d}$$
converges to $c_\gamma$ in $L^1$ and almost surely as $L\to\infty$.
\end{prop}

\begin{proof} 
We will use a similar technique to \cite{ALL}, \cite{SarnakWigman}, where an ``integral geometry sandwich'' is used.
Since we are in the discrete setting, these integrals become summations. \\

\noindent
Let $x_0=\prod_{i=1}^d\big\{-\floor[\big]{\frac{L}{2}}\big\}$ denote the lower corner of the window $W_L(0)$. Fixing $\gamma \in \mathcal{G},$ for $\ell,L\in \mathds{Z}, 0<\ell<L,$ we have the following two inequalities
(which together give a sandwich estimate for $c_{L,\gamma}$):

\begin{equation}\label{eq:sand1}
\Bigg( \frac{\floor[\big]{\frac{L}{\ell}}}{L} \Bigg)^d \frac{1}{\floor[\big]{\frac{L}{\ell}}^d}
\sum_{i_1=0}^{\floor{\frac{L}{\ell}}-1} ...
\sum_{i_d=0}^{\floor{\frac{L}{\ell}}-1}
\mathcal{N}(T^{i_1}_{e_1}...T^{i_d}_{e_d}\Phi,B_\ell(x_0);\gamma)
\leq\frac{\mathcal{N}(\Phi,W_L(0);\gamma)}{L^d},
\end{equation}
\begin{equation}\label{eq:sand2}
\Bigg( \frac{\ceil[\big]{\frac{L}{\ell}}}{L} \Bigg)^d \frac{1}{\ceil[\big]{\frac{L}{\ell}}^d}
\sum_{i_1=0}^{\ceil{\frac{L}{\ell}}-1} ...
\sum_{i_d=0}^{\ceil{\frac{L}{\ell}}-1}
\mathcal{N}^*(T^{i_1}_{e_1}...T^{i_d}_{e_d}\Phi,B_\ell(x_0);\gamma)
\geq\frac{\mathcal{N}(\Phi,W_L(0);\gamma)}{L^d},
\end{equation}
\noindent
where the maps $T^{i_1}_{e_1}...T^{i_d}_{e_d}$ are shift maps in the directions of the standard basis vectors:
$$T^{i}_{e_k}\Phi=\Big\{ x\in\mathds{R}^d : x+i\ell e_k\in\Phi \Big\}$$
\noindent
We remark that applying these shift maps to the random set $\Phi$ is equivalent to applying them in the opposite direction on $B_\ell(x_0)$, and consequently, $x_0$. We will let $S_\ell(x)\subset B_\ell(x)$ denote the cubical ``shell'' about the box of size $\ell$:
$$S_\ell(x)=\bigcup_{z} B_1(z),$$
where the union is over all $z$ having $z_i \in \left\{x_i, x_i+ \ell-1 \right\}$ for some $i=1,2,...,d$.


Notice that the volume of the shell satisfies the inequality \begin{equation}\label{eq:vol}
    \Vol(S_\ell(x))\leq2d \ell^{d-1}.
\end{equation} 
Observe that any connected component of $\Phi$ counted by $\mathcal{N^*}(\Phi,B_\ell(z);\gamma)$ but not $\mathcal{N}(\Phi,B_\ell(z);\gamma)$ intersects the box's shell, $S_\ell(z)$. 
Since the number of connected components of $\Phi$ intersecting the shell
cannot exceed the number of unit cubes in the shell, this gives us the estimate
\begin{align*}
\cN^*(\Phi,B_\ell(z);\gamma) &\leq \cN(\Phi,B_\ell(z);\gamma) + \Vol(S_\ell(x)) \\
&\leq \cN(\Phi,B_\ell(z);\gamma) + 2d \ell^{d-1},
\end{align*}
where we have also used (\ref{eq:vol}).

Therefore, we find that the expression
$$ \Bigg( \frac{\ceil[\big]{\frac{L}{\ell}}}{L} \Bigg)^d \frac{1}{\ceil[\big]{\frac{L}{\ell}}^d}
\sum_{i_1=0}^{\ceil{\frac{L}{\ell}}-1} ...
\sum_{i_d=0}^{\ceil{\frac{L}{\ell}}-1}
\mathcal{N}^*(T^{i_1}_{e_1}...T^{i_d}_{e_d}\Phi,B_\ell(x_0);\gamma)
$$
appearing in the left side of (\ref{eq:sand2}) is bounded above by
$$
\Bigg( \frac{\ceil[\big]{\frac{L}{\ell}}}{L} \Bigg)^d \frac{1}{\ceil[\big]{\frac{L}{\ell}}^d}
\sum_{i_1=0}^{\ceil{\frac{L}{\ell}}-1} ...
\sum_{i_d=0}^{\ceil{\frac{L}{\ell}}-1}
\big(\mathcal{N}(T^{i_1}_{e_1}...T^{i_d}_{e_d}\Phi,B_\ell(x_0);\gamma) + 2d \ell^{d-1}\big).
$$
We can thus utilize
\begin{equation}\label{eq:sand2b}
a(\ell) + \Bigg( \frac{\ceil[\big]{\frac{L}{\ell}}}{L} \Bigg)^d \frac{1}{\ceil[\big]{\frac{L}{\ell}}^d}
\sum_{i_1=0}^{\ceil{\frac{L}{\ell}}-1} ...
\sum_{i_d=0}^{\ceil{\frac{L}{\ell}}-1}
\mathcal{N}(T^{i_1}_{e_1}...T^{i_d}_{e_d}\Phi,B_\ell(x_0);\gamma)
\geq\frac{\mathcal{N}(\Phi,W_L(0);\gamma)}{L^d},
\end{equation}
in place of (\ref{eq:sand2})
as part of our sandwich estimate, where
$$a(\ell) := 2d\ell^{d-1} \bigg( \frac{\ceil[\big]{\frac{L}{\ell}}}{L} \bigg)^d.$$

Now we are ready to apply the ergodic theorem from \cite{MeesterRoy}
(here we state the result \cite[Prop. 2.2]{MeesterRoy}
with the added assumption of ergodicity which strengthens the conclusion as explained in \cite[Sec. 2.1]{MeesterRoy}).

\begin{thm}[Ergodic Theorem \cite{MeesterRoy}] 
Let $(\Omega,\mathcal{F},\rho)$ be a probability space, and suppose $\mathds{Z}^d$ acts on $\Omega$ via translations $T_1,T_2,...,T_d$.
If $f:\Omega \rightarrow \RR$ satisfies $f \in L^1(\rho)$
and the action of $\ZZ^d$ is ergodic, then
$$\frac{1}{n^d}\sum_{i_1=0}^{n-1} ... \sum_{i_d=0}^{n-1} f(T_1^{i_1}...T_d^{i_d}(\omega)) \to \mathds{E}(f(\omega)) \;\; a.s.\;\; as\; n\to\infty.$$ 
\end{thm}

The probability space in our application 
will be the space of site percolations $Z$,
and the measure preserving transformations are translations by a standard basis vector acting on the site percolation (this action is ergodic by \cite[Prop. 7.3]{LyPe}). 
The function $f$ that we will consider is
given by
$$f(Z)=
\mathcal{N}(\Phi,B_\ell(x_0);\gamma),
$$
where $\ell$ is fixed, and the dependence of $f$ on $Z$ is through the construction of $\Phi$.
The required integrability of $f$ is satisfied since
\begin{equation}\label{eq:integrability}
    \EE f(Z) = \EE \big( 
\mathcal{N}(\Phi,B_\ell(x_0);\gamma) \big) \leq \ell^d < \infty.
\end{equation}



\noindent
So the ergodic theorem may be applied, and as  $L\to\infty$
(which implies $\floor[\big]{\frac{L}{\ell}},\ceil[\big]{\frac{L}{\ell}} \to \infty$) we have:

\begin{equation}\label{eq:seq1}
\Bigg( \frac{\floor[\big]{\frac{L}{\ell}}}{L} \Bigg)^d \frac{1}{\floor[\big]{\frac{L}{\ell}}^d}
\sum_{i_1=0}^{\floor{\frac{L}{\ell}}-1} ...
\sum_{i_d=0}^{\floor{\frac{L}{\ell}}-1}
\mathcal{N}(T^{i_1}_{e_1}...T^{i_d}_{e_d}\Phi,B_\ell(x_0);\gamma)
\to \mathds{E}\bigg(
\frac{\mathcal{N}(\Phi,B_\ell(x_0);\gamma)}{\ell^d}\bigg) = \lambda(\ell) \; \text{ a.s.}
\end{equation}

\begin{equation}\label{eq:seq2}
\Bigg( \frac{\ceil[\big]{\frac{L}{\ell}}}{L} \Bigg)^d \frac{1}{\ceil[\big]{\frac{L}{\ell}}^d}
\sum_{i_1=0}^{\ceil{\frac{L}{\ell}}-1} ...
\sum_{i_d=0}^{\ceil{\frac{L}{\ell}}-1}
\mathcal{N}(T^{i_1}_{e_1}...T^{i_d}_{e_d}\Phi,B_\ell(x_0);\gamma)
\to \mathds{E}\bigg(
\frac{\mathcal{N}(\Phi,B_\ell(x_0);\gamma)}{\ell^d}\bigg) = \lambda (\ell) \; \text{ a.s.}
\end{equation}


The fact that the above two sequences converge to the same constant $\lambda(\ell)$ is readily seen from considering the subsequence where $\frac{L}{\ell} \in \mathds{N}$. 

Let $\e>0$. Choose $\ell \gg 0$ sufficiently large such that 
$a(\ell)<\frac{\e}{4}.$ 
Almost surely, for $L\gg \ell$ large enough we have that each of the sequences in (\ref{eq:seq1}) and (\ref{eq:seq2}) is within $\frac{\e}{4}$ of 
the limit $\lambda(\ell)$.
Applying this to the sandwich estimate provided in (\ref{eq:sand1}) and (\ref{eq:sand2b}),
we obtain:
$$\lambda(\ell)-\frac{\e}{4} \leq \frac{\mathcal{N}(\Phi,W_L(0);\gamma)}{L^d} \leq \lambda(\ell) + a(\ell) + \frac{\e}{4}
\implies \bigg|\frac{\mathcal{N}(\Phi,W_L(0);\gamma)}{L^d}-\lambda(\ell) \bigg|\leq \frac{\e}{4} + a(\ell) < \frac{\e}{2}.$$
\noindent
The remainder of the proof simply argues that $\frac{\mathcal{N}(\Phi,W_L(0);\gamma)}{L^d}$ is a Cauchy sequence, and therefore convergent, which yields the desired result. For $L,M$ sufficiently large, we have
\begin{align*}
\left|\frac{\mathcal{N}(\Phi,W_L(0);\gamma)}{L^d}-\frac{\mathcal{N}(\Phi,W_M(0);\gamma)}{M^d} \right| &\leq \left|\frac{\mathcal{N}(\Phi,W_L(0);\gamma)}{L^d}-\lambda(\ell) \right|+
\left|\frac{\mathcal{N}(\Phi,W_M(0);\gamma)}{M^d}-\lambda(\ell) \right| \\
&< \frac{\e}{2} + \frac{\e}{2} = \e.
\end{align*}
This implies
 there exists a constant $c_\gamma$ such that $\frac{\mathcal{N}(\Phi,W_L(0);\gamma)}{L^d} \to c_\gamma$ almost surely and in $L^1$.
(The convergence in $L^1$ follows from the above integrability estimate (\ref{eq:integrability}) showing that this is a dominated convergence.)
\end{proof}

\begin{remark}\label{rmk:same}
The above argument as well as the rest of the proof of Theorem \ref{thm:main} can be adapted to establish the same limit law 
while alternatively defining the measure $\mu_L$ using the counting function $\cN^*$ in place of $\cN$.
Since the discrepancy between $\cN$ and $\cN^*$ is a lower order term (as seen in the proof above), this results in the same limiting measure $\mu$.
In particular, although working with $\cN^*$ would require enlarging $\G$ to include homotopy types of infinite components, those asymptotically account for a vanishingly small portion and do not appear in the support of the limiting measure.
\end{remark}

\subsection{Positivity of $c_\gamma$}

\begin{prop}\label{prop:pos} 
For each $\gamma \in \mathcal{G}$,
$c_\gamma>0$.
\end{prop}

\begin{proof}
Fix $\gamma \in \G$,
and choose a representative $A$ 
of type $\gamma$ (i.e., $A$ is a bounded lattice component belonging to the class $\gamma$).
Let $B_\ell(x_0)$
be a box containing $A$ such that $A$ has empty intersection with the shell $S_\ell(x_0)$ of $B_\ell(x_0)$.
Let $E$ be the event that $\Phi |_{B_\ell (x_0)}=A$, and $\chi_E$ the indicator random variable for $E$. 
Using translation invariance, we have
$$\EE\mathcal{N}(T^{i_1}_{e_1}...T^{i_d}_{e_d}\Phi,B_\ell(x_0);\gamma)\geq \EE \chi_E = p^{|A|}(1-p)^{\ell^d - |A|}>0.$$
For each $c_{L,\gamma}$, we apply the bottom half (\ref{eq:sand1}) of the sandwich estimate. Then, using linearity of expectation, we have:
$$\EE c_{L,\gamma}=\EE\bigg(\frac{\mathcal{N}(\Phi,W_L(0);\gamma)}{L^d}\bigg)\geq\Bigg( \frac{\floor[\big]{\frac{L}{\ell}}}{L} \Bigg)^d \frac{1}{\floor[\big]{\frac{L}{\ell}}^d}
\sum_{i_1=0}^{\floor{\frac{L}{\ell}}-1} ...
\sum_{i_d=0}^{\floor{\frac{L}{\ell}}-1}
\EE\mathcal{N}(T^{i_1}_{e_1}...T^{i_d}_{e_d}\Phi,B_\ell(x_0);\gamma).
$$
Applying our first inequality yields:
$$\EE c_{L,\gamma} \geq\Bigg( \frac{\floor[\big]{\frac{L}{\ell}}}{L} \Bigg)^d \frac{1}{\floor[\big]{\frac{L}{\ell}}^d}
\sum_{i_1=0}^{\floor{\frac{L}{\ell}}-1} ...
\sum_{i_d=0}^{\floor{\frac{L}{\ell}}-1} p^{|A|}(1-p)^{\ell^d - |A|} = \Bigg( \frac{\floor[\big]{\frac{L}{\ell}}}{L} \Bigg)^d  (p^{|A|}(1-p)^{\ell^d - |A|}).
$$
Now taking the limit as $L$ tends to infinity gives:
$$c_\gamma = \lim_{L\to\infty}\big(\EE c_{L,\gamma}\big) \geq \lim_{L\to\infty} \Bigg( \Bigg( \frac{\floor[\big]{\frac{L}{\ell}}}{L} \Bigg)^d  (p^{|A|}(1-p)^{\ell^d - |A|}) \Bigg) = \frac{1}{\ell^d}(p^{|A|}(1-p)^{\ell^d - |A|}) >0,$$
as desired.
\end{proof}

\subsection{Limit law for coefficients}

Proposition \ref{prop:const}
establishes the convergence 
$c_{L,\gamma} \rightarrow c_\gamma$
in probability as $L \rightarrow \infty$.
Next we use this to verify the convergence
of the coefficients
$$a_{L,\gamma} = \frac{\cN(\Phi,W_L(0),\gamma)}{b_0(\Phi_L)}$$
appearing in the definition (\ref{eq:mudef}) of the measure $\mu_L$.

\begin{prop}\label{prop:coeff} 
For each $\gamma \in \G$,
the coefficient $a_{L,\gamma}$
converges in probability
to a positive deterministic constant $a_\gamma$ as $L\rightarrow \infty$.
\end{prop}

\begin{proof}
Define
$$c_L := \frac{b_0(\Phi_L)}{L^d}, \quad \quad c := \lim_{L\to \infty} c_L.$$
Note that the existence of the limit defining $c$ follows from applying the same proof as in Proposition \ref{prop:const} to the sequence $c_L$ (thus counting components with no restriction on homotopy type). 
We write the coefficient
$a_{L,\gamma}$ as:
\begin{align}
a_{L,\gamma} &= \frac{L^d}{b_0(\Phi_L)}\frac{\cN(\Phi,W_L(0),\gamma)}{L^d}\\
 &= \frac{c_{L,\gamma}}{c_L},
\end{align}
By the convergence statements in Proposition \ref{prop:const},
and since $c>0$
(which follows from Proposition \ref{prop:pos} since $c \geq c_\gamma$),
we have $a_{L,\gamma}=\frac{c_{L, \gamma}}{c_L}$ converges in $L^1$ to a constant $a_\gamma = \frac{c_\gamma}{c}$ as $L \rightarrow \infty$, and the positivity of the coefficients $a_\gamma$ follows from the positivity of $c_\gamma$ stated in Proposition \ref{prop:pos}.
\end{proof}

\subsection{The homotopy limit law: Finishing the proof of Theorem \ref{thm:main}}\label{sec:main}

Finally, we upgrade the mode of convergence (from  coefficient-wise convergence to convergence in probability respecting total variation distance) in order to establish the limit law for the sequence of measures $\mu_L$.

Proposition \ref{prop:coeff}
establishes the convergence 
$a_{L,\gamma} \rightarrow a_\gamma$
in probability as $L \rightarrow \infty$ of each of the coefficients $a_{L,\gamma}$
appearing in the measure $\mu_L$.
Let 
$$\mu = \sum a_\gamma \delta_{\gamma}$$
denote the measure constructed from the list of limiting coefficients provided by Proposition \ref{prop:coeff}.

In order to prove convergence in probability 
$\mu_L \rightarrow \mu$
of the sequence of measures, we need the following lemma.

\begin{lemma}[Topology does not leak to infinity]\label{lemma:notleak}
For every $\delta> 0$
there exists a finite set $g \subset \G$
and $L_0>0$ such that for all $L \geq L_0$
$$ \EE \sum_{\gamma \in g^c} c_{L,\gamma} < \delta.$$
\end{lemma}

\begin{proof}
Let $\delta > 0$
be arbitrary.
We prove the (stronger) deterministic statement that there exists a
finite set $g \subset \G$
such that 
$$ \sum_{\gamma \in g^c} c_{L,\gamma} < \delta.$$
holds.
Choose an integer $S > 1 / \delta$, and let $g \subset \G$ denote the set of
homotopy types that have at least one representative with size
no larger than $S$.
This set is finite, and
we have 
\begin{align*}
\sum_{\gamma \in g^c} c_{L,\gamma} &=  \sum_{\gamma \in g^c} \frac{ \cN(\Phi,W_L(0);\gamma)}{L^d} \\
&\leq \frac{L^d/S}{L^d} = \frac{1}{S} 
< \delta,
\end{align*}
as desired.
\end{proof}

The rest of the proof
that $\mu_L \rightarrow \mu$ in probability, and that $\mu$ is indeed a probability measure,
is purely measure theoretic and now follows from \cite[Lemma 2.4, Proposition 3.4]{ALL}.

The statement that the support of $\mu$ is all of $\G$ follows from the positivity of $a_\gamma$ stated in Proposition \ref{prop:coeff}.

\section{Dependence of $\mu$ on $p$}\label{sec:perc}

In this section we will occasionally utilize the same notation $c_{L,\gamma}, c_\gamma, c, a_\gamma$ that appeared throughout Section \ref{sec:main} in the construction of $\mu$.

\subsection{Exponential tail for $p$ below the percolation threshold}\label{sec:exptail}

In this section, we prove Theorem \ref{thm:exptailmu}, essentially as a consequence of the following result from percolation theory \cite{Grimmett}.

Let $C_0$ denote the (possibly empty) connected component of $\Phi$ that contains the interior of the cube $B_1(0)$, and let $|C_0|$ denote its size (the number of unit cubes in $C_0$).

\begin{thm}\label{thm:exptailperc}
Suppose $0<p<p_c$.
There exists $\lambda = \lambda(p)$ such that
\begin{equation}\label{eq:exptailperc}
\PP \{ |C_0| \geq n \} \leq e^{- \lambda \cdot n}, \quad \text{ for all } n \geq 1.
\end{equation}
\end{thm}

Note that the result in \cite[Thm. 6.75]{Grimmett} is actually stated for bond percolation (as opposed to site percolation considered here) on a cubical lattice, but the same proof applies \emph{mutatis mutandis} to the setting of site percolation.

\begin{proof}[Proof of Theorem \ref{thm:exptailmu}]
As explained in Remark \ref{rmk:cubicalcomplex} the set $\Phi$ is homotopy equivalent to a cubical complex.
The connected component corresponding to $C_0$ has $|C_0|$ many vertices.  The number of $k$-dimensional faces incident on each vertex is bounded by $2^k\binom{d}{k}$.
This gives the rough estimate
$\beta_k(C_0) \leq 2^k \binom{d}{k} |C_0|$
for the $k$th Betti number.  Here we are
estimating $\beta_k(C_0)$ by the dimension of the span of $k$-cycles (without considering the kernel of the boundary operator), and we are estimating the dimension of the span of $k$-cycles by the number of $k$-faces (the $k$-faces provide a spanning set for the space of $k$-cycles).


Hence, 
$$\PP \{ \beta_k(C_0) \geq n \} \leq
\PP \left\{ |C_0| \geq c(k,d) \cdot n \right\},$$
where $c(k,d):=2^{-k}\binom{d}{k}^{-1}$,
and an application of Theorem \ref{thm:exptailperc} gives
\begin{equation}\label{eq:decaylemma}
\PP \{ \beta_k(C_0)  \geq n \} \leq e^{- \lambda \cdot c(k,d) \cdot   n}, \quad \text{ for all } n \geq 1.
\end{equation}
We will apply this to estimate the measure $\mu(T_n)$ of the tail set $T_n$. Recall its definition:
\begin{equation}
T_n := \{ \gamma \in \G : \beta_k(\gamma) \geq n \}.
\end{equation}
First we write
\begin{align*}
\mu(T_n) = \sum_{\gamma \in T_n} a_\gamma &= \frac{1}{c} \sum_{\gamma \in T_n} c_\gamma \\
&= \frac{1}{c} \sum_{\gamma \in T_n} \lim_{L \rightarrow \infty} \EE c_{L,\gamma} \\
&\leq \frac{1}{c} \liminf_{L \rightarrow \infty} \sum_{\gamma \in T_n}  \EE c_{L,\gamma},
\end{align*}
where we have used Fatou's lemma in the last step.
Using the monotone convergence theorem to reverse the order of summation and expectation we then obtain
\begin{align*}
\frac{1}{c} \liminf_{L \rightarrow \infty} \sum_{\gamma \in T_n}  \EE c_{L,\gamma} &= \frac{1}{c} \liminf_{L \rightarrow \infty} \EE \sum_{\gamma \in T_n}  c_{L,\gamma} \\
&= \frac{1}{c} \liminf_{L \rightarrow \infty} \EE \frac{\cN(\Phi,W_L(0);T_n)}{L^d},
\end{align*}
where we use $\cN(\Phi,W_L(0);T_n)$
to denote the number of components
completely contained in $W_L(0)$ with type belonging to $T_n$.
We next use the rough estimate
\begin{align*}
\EE \cN(\Phi,W_L(0);T_n) &\leq  \sum \PP \{ \beta_k(C_x) \geq n \} \\
&= L^d \PP \{ \beta_k(C_0) \geq n \},
\end{align*}
where $C_x$ denotes the connected component of $\Phi$ containing the interior of the unit cube $B_1(x)$, the sum is over all points $x$ such that $B_1(x) \subset W_L(0)$,
and the second line follows from translation invariance.
Combining this with the above estimate for $\mu(T_n)$ we obtain
$$\mu(T_n) \leq \frac{1}{c} \PP \{ \beta_k(C_0) \geq n \},$$
and the desired exponential decay now follows from (\ref{eq:decaylemma}).
\end{proof}

\begin{remark}\label{rmk:mu0}
We note that the above proof of Theorem \ref{thm:exptailmu} also gives information on another interesting probability measure associated with homotopy types, the probability measure $\mu_0$ for the random variable given by the homotopy type of the component $C_0$.  Namely, we have that when $p<p_c$ the probability measure $\mu_0$ satisfies the exponential tail estimate given in \eqref{eq:decaylemma}.
\end{remark}

\subsection{Sub-exponential tail for $p$ above the percolation threshold}\label{sec:subexptail}

In this section we prove Theorem \ref{thm:subexptailmu} by
adapting methods from percolation theory
\cite[Sec. 8.6]{Grimmett}
while combining them with topological considerations.

The lower bound in Theorem \ref{thm:subexptailmu}
will be established while focusing attention on special clusters, referred to here as ``egg sac clusters'',
that possess a particular structure consisting of a large double-layer wall (inner layer completely black, outer layer completely white) surrounding an abundance of ``eggs'' attached to the inner layer of the wall by ``umbilical paths'', with each egg consisting of a completely black shell surrounding an ``embryonic $k$-cycle'' attached to the egg shell by a single path (see Figure \ref{fig:eggsac}).

Let us make this terminology precise.
For fixed $k$, we will refer to the part of $\Phi$ in the box $B_7(x)$ as an \emph{egg} if
the following conditions are satisfied:

\begin{itemize}
    \item $\Phi \cap B_7(x)$ is connected,
    \item the cubes in the shell $S_7(x)$ are all black,
    \item the set
 $\Phi \cap B_3(x+(2,2,...,2))$ has a single nontrivial $k$-cycle,
    \item the inner shell $S_5(x+(1,1,...,1)$ has exactly one black cube.
\end{itemize}
In that case we will refer to the $k$-cycle as an \emph{embryonic $k$-cycle}.
The surrounding structure of the egg will serve two purposes: (i) to insulate the embryonic $k$-cycle and prevent it from becoming a boundary and (ii) to
facilitate an estimate on the probability of the event that a given egg is attached to the sac wall by an umbilical path.

\begin{figure}[ht]
    \centering
    \includegraphics[scale=0.34]{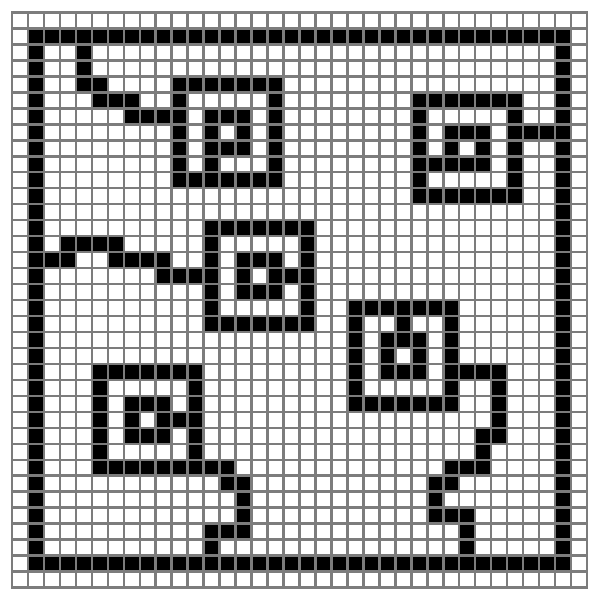}
    \caption{Anatomy of an egg sac cluster (idealized rendition conveying only the essential features): A finite cluster with large $k$th Betti number formed by a double-layer wall (with inner layer completely black, outer layer completely white) surrounding an abundance of eggs connected to the inner layer of the wall by umbilical paths. Each egg includes an embryonic $k$-cycle.
    }
    \label{fig:eggsac}
\end{figure}

Take the covering $\Phi\cap B_5(x+(1,1,...,1)$ and $\Phi\cap B_3(x+(2,2,...,2)^c$. We may apply the Mayer-Vietoris sequence to this covering, and, using the condition that $S_5(x+(1,1,...,1))$ has exactly one (contractible) black cube, we conclude that the embryonic $k$-cycle makes a positive contribution to the $k$th Betti number of the connected component it belongs to (cf. \cite[Sec. 2.2, Exercise 31]{Hatcher}). This will allow us to estimate from below the $k$th Betti number of an egg sac cluster by the number of eggs it has.

We note in passing that in the case $k=d-1$ one could simply work with embryonic $k$-cycles throughout this section (relying on Alexander duality) and omit the additional insulating structure of the egg.

Fix a parameter $0<\rho<1$.
Given a box $B_L(x)$
we will say that it \emph{has an egg sac cluster with abundance rate $\rho$} if the following conditions are satisfied
\begin{itemize}
    \item The outer layer $S_L(x)$ is completely white.
    \item The inner layer
    $S_{L-2}(x+(1,1,...,1))$
    is completely black.
    \item The number of eggs in 
    $B_{L-4}(x+(2,2,...,2))$ that are connected to $S_{L-2}(x+(1,1,...,1))$
    by a path (referred to as an \emph{umbilical path}) through $\Phi$ is at least $\rho \cdot L^d$.
\end{itemize}

The condition that the outer layer $S_L(x)$ of the egg sac wall is completely empty ensures that the egg sac cluster is finite.
The condition that the inner layer
    $S_{L-2}(x+(1,1,...,1))$
of the wall is completely black ensures that the eggs connected to this layer by umbilical paths are all part of the same connected component.

\begin{lemma}\label{lemma:egg}
Let $\phi(p)$ denote the
probability that the
box $B_7(0)$ has an egg that is connected to $\{\infty\}$ by a path through $\Phi$.
Then $\phi(p)>0$ whenever $p_c < p < 1$.
\end{lemma}

\begin{proof}[Proof of Lemma \ref{lemma:egg}]
We will show that
\begin{equation}\label{eq:philower}
\phi(p) \geq p^{7^d}(1-p)^{7^d}\theta(p),    
\end{equation} 
where $\theta(p)$ denotes the probability
that $|C_0| = \infty$.
This will establish the desired positivity of $\phi(p)$
since $\theta(p)>0$ for $p_c < p < 1$
(by definition of $p_c$).

Let us view the event in the statement of the lemma as an intersection
of the event $E_1$ that $B_7(0)$ is an egg
and the event $E_2$ that there is a path to $\{ \infty \}$ through the complement of $B_7(0)$ starting from some cube adjacent to $B_7(0)$, i.e., a cube in $S_9((-1,-1,...,-1))$ that shares a face with one of the cubes in $S_7(0)$.
If the events $E_1$ and $E_2$ both occur then the egg provided by $E_1$ is connected to $\{ \infty \}$ by a path through $\Phi$, since the occurrence of $E_1$ means the shell $S_7(0)$ is completely black, so that the path provided by the event $E_2$ can be extended to produce a path to $\{\infty\}$ that begins in the shell $S_7(0)$.
The events $E_1$ and $E_2$ are independent since they concern disjoint sets of cubes in the coloring of the cubical grid.  

In order to give a rough estimate for the probability of the event $E_1$, choose a specific coloring of $B_7(0)$  verifying the conditions to be an egg and let $b$ and $w$ denote the number of black and white cubes respectively.
Then we have 
\be\label{eq:E1}
P(E_1) \geq p^b(1-p)^w > p^{7^d}(1-p)^{7^d}.
\ee

Next notice that the probability of the event $E_2$ satisfies
\be\label{eq:E2}
\PP (E_2) \geq \theta(p),
\ee
since the occurrence of the event $|C_0| = \infty$ means there exists a path $P$ through $\Phi$ from
$B_1(0)$ to $\{ \infty\}$
which, by considering an appropriate truncation of $P$, implies the event $E_2$.
Indeed, truncating the path $P$,
keeping the part after the step where $P$ last exits $B_7(0)$, provides a path to $\{ \infty \}$ through the complement of $B_7(0)$ starting from some cube adjacent to the box $B_7(0)$ as described in the event $E_2$.

The two estimates \eqref{eq:E1}, \eqref{eq:E2} together with the independence of the events $E_1$ and $E_2$ establish (\ref{eq:philower}) and prove the lemma.
\end{proof}

Let $X_L$ denote the number of eggs contained in the box $B_{L-4}((2,2,...,2))$
that have a path through
$\Phi$ to $S_{L-4}((2,2,...,2))$.

\begin{lemma}\label{lemma:eggcount}
Suppose that $p_c<p<1$.
Then for $L$ sufficiently large
$$\PP \left\{ X_L \geq \frac{\phi(p)}{2} L^d \right\} \geq \frac{\phi(p)}{3}.$$ 
\end{lemma}

\begin{proof}[Proof of Lemma \ref{lemma:eggcount}]
Let $R_L$ denote
the number of eggs contained in $B_{L-4}((2,2,...,2))$ that are joined to $\{\infty\}$
by a path in $\Phi$.
It is clear that $X_L \geq R_L$, and hence in order to prove the lemma it suffices to show that
\begin{equation}\label{eq:R_Lprob}
\PP \left\{ R_L \geq \frac{\phi(p)}{2} L^d \right\} \geq \frac{\phi(p)}{3}.
\end{equation}

Writing $R_L$ as a sum of indicator variables,
where the sum is over all boxes $B_7(x_k)$
contained in $B_{L-4}((2,2,...,2))$,
we compute the average
\begin{equation}\label{eq:R_Lexp}
    \EE R_L = \phi(p) (L-11)^d \geq \frac{5}{6} \phi(p) L^d,
\end{equation}
where the inequality holds for all $L$ sufficiently large.
The following estimate for the expectation $\EE R_L$ comes from considering the
event appearing in (\ref{eq:R_Lprob}) and its complementary event.
\begin{align*}
\EE R_L &\leq L^d \PP \left\{ R_L \geq \frac{\phi(p)}{2}L^d \right\} + \frac{\phi(p)}{2} L^d \PP \left\{ R_L \leq \frac{\phi(p)}{2}L^d \right\} \\
&\leq L^d \PP \left\{ R_L \geq \frac{\phi(p)}{2}L^d \right\} + \frac{\phi(p)}{2} L^d.
\end{align*}
Together with (\ref{eq:R_Lexp})
this gives
\begin{equation}
\frac{\phi(p)}{3} L^d \leq L^d \PP \left\{ R_L \geq \frac{\phi(p)}{2}L^d \right\},
\end{equation}
which implies (\ref{eq:R_Lprob}) and proves the lemma.
\end{proof}

\begin{prop}\label{prop:subexptailmu}
Suppose $p_c<p<1$. 
There exists $0 < \rho = \rho(p) < 1$ and $\eta_0 = \eta_0(p) > 0$ such that
\begin{equation}\label{eq:subexptail}
\PP \{ B_{L}(0) \text{ has an egg sac cluster with abundance rate $\rho$} \} \geq \exp\left\{ - \eta_0 \cdot L^{(d-1)} \right\},
\end{equation}
for all $L$ sufficiently large.
\end{prop}

\begin{proof}[Proof of Proposition \ref{prop:subexptailmu}]
We take $\rho = \phi(p)/2$.
The event 
\begin{equation}\label{eq:EL} E_L:= \{ B_{L}(0) \text{ has an egg sac cluster with abundance rate $\rho$} \} \end{equation}
appearing in \eqref{eq:subexptail} 
contains the intersection of the event
$$V_1 := \left\{ X_L \geq \frac{\phi(p)}{2} L^d \right\},$$
and the event
$V_2$ that the shell $S_{L-2}((1,1,...,1)$
is completely black and the shell
$S_{L}(0)$ is completely white.


Using independence of the events $V_1$ and $V_2$ we estimate
\begin{align*}
\PP E_L &\geq \PP V_1 \cap V_2 \\
 &= \PP V_1 \PP V_2 \\
 &\geq \frac{\phi(p)}{3} \PP V_2,
\end{align*}
where we applied Lemma \ref{lemma:eggcount} in the last line of the above.
Combining this with the elementary estimate,
\begin{align*}
\PP V_2 &\geq p^{2dL^{d-1}}(1-p)^{2dL^{d-1}} \\
&= [p^{2d}(1-p)^{2d}]^{L^{d-1}},
\end{align*}
proves the desired subexponential lower bound, which can be written as
$$\PP E_L \geq \frac{\phi(p)}{3} \exp\left\{ - \eta_1 L^{d-1} \right\}, \quad \eta_1 = -\log [p^{2d}(1-p)^{2d}],$$
and by replacing $\eta_1$ by a slightly larger  $\eta_0$ we may omit the constant factor to arrive at (\ref{eq:subexptail})
for all $L$ sufficiently large.
\end{proof}

\begin{proof}[Proof of Theorem  \ref{thm:subexptailmu}]
We wish to estimate the measure $\mu(T_n)$ of the tail set $T_n = T_n(k,d)$, which we recall is defined as
\begin{equation}
T_n := \{ \gamma \in \G : \beta_k(\gamma) \geq n \}.
\end{equation}
First, we write (again utilizing the notation and results from Section \ref{sec:coeff})
\begin{align*}
\mu(T_n) = \sum_{\gamma \in T_n} a_\gamma &= \frac{1}{c} \sum_{\gamma \in T_n} c_\gamma \\
&= \frac{1}{c} \sum_{\gamma \in T_n} \lim_{L \rightarrow \infty} \EE c_{L,\gamma} \\
&\geq \frac{1}{c} \sum_{\gamma \in S_n} \lim_{L \rightarrow \infty} \EE c_{L,\gamma},
\end{align*}
where $S_n \subset T_n$ is the finite set of types $\gamma \in T_n$ that can be realized by some connected cubical complex contained in $B_\ell(0)$, where 
$\ell = \lceil (n/\rho)^{1/d} \rceil$ with $\rho$ generated by applying Proposition \ref{prop:subexptailmu}.
We can place the limit and expectation outside the summation in the last line of the above displayed equations since the set $S_n$ is finite.  Then we have
\begin{align*}
\frac{1}{c} \sum_{\gamma \in S_n} \lim_{L \rightarrow \infty} \EE c_{L,\gamma}
&= \frac{1}{c} \lim_{L \rightarrow \infty} \EE \sum_{\gamma \in S_n}  c_{L,\gamma} \\
&= \frac{1}{c} \lim_{L \rightarrow \infty} \EE \frac{\cN(\Phi,W_L(0);S_n)}{L^d},
\end{align*}
where $\cN(\Phi,W_L(0);S_n)$ denotes the number of components
completely contained in $W_L(0)$ with type belonging to $S_n$. 
Hence, from the above two sets of displayed lines we have the estimate
\be\label{eq:T_nLB}
\mu(T_n) \geq \frac{1}{c} \lim_{L \rightarrow \infty} \EE \frac{\cN(\Phi,W_L(0);S_n)}{L^d}.
\ee
If the box $B_{\ell}(x)$ contains a connected component of $\Phi$ that is an egg sac cluster with abundance rate $\rho$, then the $k$th Betti number $\beta_k$ of that cluster is at least $n$ so that its homotopy type is contained in $T_n$, and hence also in $S_n$ (since the cluster is contained in a box of size $\ell$).
Using this observation, we next estimate
\begin{align*}
\EE \cN(\Phi,W_L(0);S_n) &\geq \sum \PP \{ B_\ell(x) \text{ has an egg sac cluster with abundance $\rho$} \} \\
&= (L-\ell)^d \PP \{ E_\ell \},
\end{align*}
where the sum is over all points $x$ such that $B_\ell(x) \subset W_L(0)$,
the second line follows from translation invariance,
and $E_\ell$ denotes the event for the case $x=0$ as defined in (\ref{eq:EL}).
Combining this with the above estimate \eqref{eq:T_nLB} for $\mu(T_n)$
and applying Proposition \ref{prop:subexptailmu}, we obtain
\begin{align*}
\mu(T_n) &\geq \frac{1}{c} \PP \{ E_\ell \} \\
&\geq \frac{1}{c} \exp \left\{ -\eta_0 \ell^{d-1} \right\} \\
&\geq \frac{1}{c}  \exp \left\{ -\frac{\eta_0 2^{d-1}}{\rho^{(d-1)/d}} n^{(d-1)/d} \right\},
\end{align*}
which proves the statement in the theorem.
\end{proof}

\begin{remark}
Continuing the point made in Remark \ref{rmk:mu0}, we note that our results in the supercritical regime also give information on 
the probability measure $\mu_0$ for the distribution of homotopy types of the random cluster $C_0$.
Namely, it follows from Proposition \ref{prop:subexptailmu}, that $\mu_0$ has a subexponential tail for $p > p_c$.
Indeed, by translation invariance, we may apply the proposition while replacing the box $B_L(0)$ with $B_L(-1,-1,...,-1)$, so that the associated egg sac cluster coincides with the connected component $C_0$.  Then \eqref{eq:subexptail} provides a subexponential lower bound on the tail decay of $\mu_0$.
\end{remark}

\subsection{Continuity of $c_\gamma$ with respect to $p$}\label{sec:continuity}

Recall that in \cite{HiTs} Hiraoka and Tsunoda proved uniform continuity (with respect to the coloring probability $p$) of each of the limiting Betti numbers defined as follows: 
$$\hat{\beta}_{q}(p):= \limsup_{n\to\infty} \frac{\EE[\beta_q(X(p)\cap W_{2n}(0)]}{|W_{2n}(0)|}, \quad p \in [0,1],$$ 
 where $\beta_q(X(p)\cap W_{2n}(0))$ is the $q$th Betti number of the restriction to $W_{2n}$ of a random cubical complex $X(p)$ that is defined in \cite{HiTs} and is analogous to $\Phi$. Using a similar method, we can show continuity of each coefficient $c_\gamma$ in the homotopy limit law for $\Phi$.

\begin{prop}
Fix a homotopy type $\gamma\in\G$,
and let $c_\gamma$ be as defined in Proposition \ref{prop:const}. 
As a function of the coloring probability $p \in [0,1]$,
$c_\gamma$ is Lipschitz continuous.
\end{prop}

\begin{proof}

Let $\hat{f}\colon \ZZ^d\to [0,1]$,
with the point evaluations $\hat{f}(z)$ being independent random variables sampled uniformly from $[0,1]$. Note that $\hat{f}$'s sub-level filtration forms a ``coupling" for the random variables, $Z(p)$. Recall $Z(p)$ denotes the random set of colored points in $\ZZ^d$. Namely, for each $p\in [0,1]$, the sublevel set $\hat{f}^{-1}([0,p])$ has the same distribution of probability as $Z(p)$. 

Let $L>0$. For any $0 \leq p_1 < p_2 \leq 1$, we have (utilizing the notation $\cN(\cdot,\cdot;\cdot)$ in Definition \ref{def:N}):
\be\label{eq:estchange}
|\mathcal{N}(\Phi(p_1),W_L(0);\gamma) - \mathcal{N}(\Phi(p_2),W_L(0);\gamma)| \leq 2d\sum_{b\in W_L(0)} I_{\{p_1<\hat{f}(b)\leq p_2\}},
\ee
where $I_{\{p_1<\hat{f}(b)\leq p_2\}}$ denotes the indicator function.
Indeed, the sum in the right hand side of (\ref{eq:estchange}) counts the number of new cubes that get colored from $p_1$ to $p_2$, and each new cube can change the count $\cN(\Phi,W_L(0);\gamma)$ by at most $2d$, the number of faces of an individual cube. To change a component's homotopy type requires the new cube to be adjacent to the component  on one of its $2d$ faces.
Now we normalize by $L^d$ and apply expectation in (\ref{eq:estchange}) to get:
\begin{align*}
\EE\bigg|\frac{\mathcal{N}(\Phi(p_1),W_L(0);\gamma)}{L^d} - \frac{\mathcal{N}(\Phi(p_2),W_L(0);\gamma)}{L^d}\bigg| &\leq \EE\bigg( \frac{2d}{L^d} \sum_{b\in W_L(0)} I_{\{p_1<\hat{f}(b)\leq p_2\}} \bigg) \\
&= 2d|p_2 - p_1|.
\end{align*}

Combining this with the estimate
\begin{align*}
|c_{L,\gamma}(p_1) - c_{L,\gamma}(p_2)| &= \bigg|\EE\bigg(\frac{\mathcal{N}(\Phi(p_1),W_L(0);\gamma)}{L^d}\bigg) - \EE\bigg(\frac{\mathcal{N}(\Phi(p_2),W_L(0);\gamma)}{L^d}\bigg)\bigg| \\
&\leq \EE\bigg|\frac{\mathcal{N}(\Phi(p_1),W_L(0);\gamma)}{L^d} - \frac{\mathcal{N}(\Phi(p_2),W_L(0);\gamma)}{L^d}\bigg|
\end{align*}
gives
$$|c_{L,\gamma}(p_1) - c_{L,\gamma}(p_2)| \leq 2d |p_2-p_1|,$$
and taking the limit as $L\to\infty$ yields the desired Lipschitz continuity for $c_\gamma$. \end{proof}


\subsection{Empirical results}\label{sec:empirical}

For the computational results, we restrict ourselves to $d=2$ dimensions. In this case, $\mathcal{G}$ consists of precisely the homotopy types of finite wedges of circles $S^1$:
$$\text{For }k\in\NN, \; \gamma(k) := \bigg[\bigvee_{i=1}^k S^1 \bigg], \quad \gamma(0):=[\{(0,0)\}].$$ The goal is to compute approximations of the coefficients in the limiting homotopy measure (however, we will consider a modified version of the homotopy count explained below). To do this, we sample a number of colorings $\Phi$ and analyze $\frac{\mathcal{N}(\Phi,W_L(0);\gamma)}{L^d}$ for some large $L$. We will use MatLab and the homology software Perseus \cite{perseus}
to detect these homotopy types, since homology is sufficient to sort wedges of $S^1$.

Our algorithm works as follows. First, fix a large $L$, and populate each element of a $L\times L$ matrix $1$ with probability $p$, and $0$ otherwise. This represents the lattice of points in $\ZZ^2$ colored black or white. Then, we perform a breadth-first search algorithm on untested grid elements, until we have isolated each connected component. The component is then passed to the Perseus software to determine $\beta_1$ of the component, which is sufficient to classify the homotopy type. We will apply this algorithm several times for a mesh of probabilities on $p\in(0,1)$. Finally, we aggregate the homotopy counts.

\begin{table}[h!]
\begin{center}
\begin{tabular}{|cl|l|l|l|l|l|l|l|}
\hline
\multicolumn{9}{|c|}{Values of $100\hat{c}_{\gamma(k)}(p)$}                                \\ \hline
\multicolumn{1}{|l}{}                    & \multicolumn{8}{c|}{$k$}                                        \\ \cline{2-9} 
\multicolumn{1}{|c|}{\multirow{6}{*}{$p$}} &  & 0 & 1 & 2 & 3 & 4 & 5 & 10 \\ \cline{2-9} 
\multicolumn{1}{|c|}{}                   & 0.25  & 12.8951 & 0.0168 & 0.0004 & 0.0000 & 0.0000 & 0.0000 & 0.0000 \\ \cline{2-9} 
\multicolumn{1}{|c|}{}                   & 0.40  & 10.4106 & 0.2170 & 0.0462 & 0.0130 & 0.0046 & 0.0019 & 0.0001 \\ \cline{2-9} 
\multicolumn{1}{|c|}{}                   & 0.55  & 4.1136  & 0.1415 & 0.0619 & 0.0370 & 0.0253 & 0.0198 & 0.0062 \\ \cline{2-9} 
\multicolumn{1}{|c|}{}                   & 0.5920  & 2.7856  & 0.0720 & 0.0384 & 0.0096 & 0.0048 & 0.0016 & 0.0032 \\ \cline{2-9} 
\multicolumn{1}{|c|}{}                   & 0.5925  & 2.6880  & 0.0704 & 0.0256 & 0.0080 & 0.0064 & 0.0080 & 0.0032 \\ \cline{2-9} 
\multicolumn{1}{|c|}{}                   & 0.5930  & 2.5360  & 0.0608 & 0.0176 & 0.0128 & 0.0096 & 0.0064 & 0.0016 \\ \cline{2-9} 
\multicolumn{1}{|c|}{}                   & 0.70  & 0.7734  & 0.0025 & 0.0006 & 0.0002 & 0.0002 & 0.0000 & 0.0000 \\ \cline{2-9} 
\multicolumn{1}{|c|}{}                   & 0.85  & 0.0498  & 0.0000 & 0.0000 & 0.0000 & 0.0000 & 0.0000 & 0.0000 \\ \hline
\end{tabular}
\label{table:1}
\caption{
A table of approximations for $100\hat{c}_{\gamma(k)(p)}$ on a $250\times 250$ grid averaged over 20 iterates of the algorithm. The limited window size gives the zero values for high $k$, even though Proposition \ref{prop:pos} states that $\hat{c}_{\gamma(k)}(p)>0$ for any choice of $ p\in (0,1)$ and $ k\in \NN$. The values are multiplied by $100$ to counteract the strong normalization by window volume. 
}
\end{center}
\end{table}

\begin{table}[ht]
\begin{center}
\begin{tabular}{|cl|l|l|l|l|l|l|l|}
\hline
\multicolumn{9}{|c|}{Values of $\hat{a}_{\gamma(k)}(p)$}                                                                              \\ \hline
\multicolumn{1}{|l}{}                    & \multicolumn{8}{c|}{$k$}                                        \\ \cline{2-9} 
\multicolumn{1}{|c|}{\multirow{6}{*}{$p$}} &  & 0 & 1 & 2 & 3 & 4 & 5 & 10 \\ \cline{2-9} 
\multicolumn{1}{|c|}{}                   & 0.25  & 0.9987 & 0.0013 & 0.0000 & 0.0000 & 0.0000 & 0.0000 & 0.0000 \\ \cline{2-9} 
\multicolumn{1}{|c|}{}                   & 0.40  & 0.9735 & 0.0203 & 0.0043 & 0.0012 & 0.0004 & 0.0002 & 0.0000 \\ \cline{2-9} 
\multicolumn{1}{|c|}{}                   & 0.55  & 0.9088 & 0.0313 & 0.0137 & 0.0082 & 0.0056 & 0.0044 & 0.0014 \\ \cline{2-9} 
\multicolumn{1}{|c|}{}                   & 0.5920  & 0.9396  & 0.0243 & 0.0130 & 0.0032 & 0.0016 & 0.0005 & 0.0011 \\ \cline{2-9} 
\multicolumn{1}{|c|}{}                   & 0.5925  & 0.9344  & 0.0245 & 0.0089 & 0.0028 & 0.0022 & 0.0028 & 0.0011 \\ \cline{2-9} 
\multicolumn{1}{|c|}{}                   & 0.5930  & 0.9451  & 0.0227 & 0.0066 & 0.0048 & 0.0036 & 0.0024 & 0.0006 \\ \cline{2-9} 
\multicolumn{1}{|c|}{}                   & 0.70  & 0.9933  & 0.0032 & 0.0008 & 0.0002 & 0.0002 & 0.0000 & 0.0000 \\ \cline{2-9} 
\multicolumn{1}{|c|}{}                   & 0.85  & 0.9688  & 0.0000 & 0.0000 & 0.0000 & 0.0000 & 0.0000 & 0.0000 \\ \hline
\end{tabular}
\label{table:2}
\caption{
A table of approximations for $\hat{a}_{\gamma(k)(p)}$ on a $250\times 250$ grid averaged over 20 iterates of the algorithm. Again, the limited window size gives the zero values for high $k$, despite Proposition \ref{prop:pos}. 
}
\end{center}
\end{table}

Perseus computes the homology of the unions of closed cubes rather than the interior of unions of closed cubes. 
Thus, while the algorithm ignores corner connections while separating components, it includes corner connections while computing $\beta_1$ for each component.
The proof in Section \ref{sec:coeff} can easily be adapted to show the existence of a limiting homotopy measure $\hat{\mu}$ based on this hybrid method of counting.  Let  $\hat{a}_{\gamma_k}$ denote the coefficients in such limiting measure $\hat{\mu}$ and $\hat{c}_{\gamma(k)}$ denote the associated coefficients analogous to $c_{\gamma(k)}$ (based on normalizing by window size rather than total number of components). 
While the numerical values of the coefficients $\hat{c}_{\gamma(k)}$ will certainly differ from those of $c_{\gamma(k)}$, we expect that the two lists of coefficients share qualitative features such as the interesting trends depicted in the figures below.

\begin{figure}[ht]
    \centering
    \includegraphics[scale=0.45]{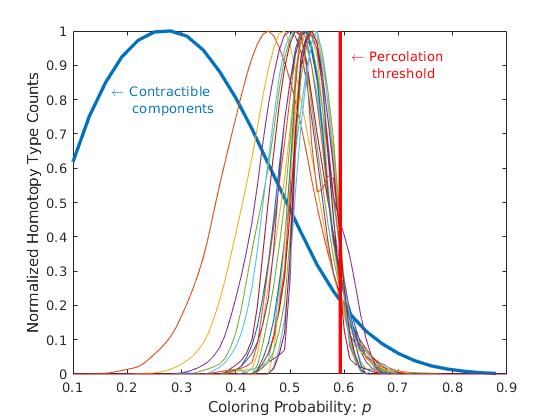}
    \caption{Each curve is a plot of the number of connected components with homotopy type a wedge of $k$ circles, with $k$ ranging from $0$ to $20$.  These have been normalized so that the maxima all have the same height.
    The positions where the maxima occur appears to form an increasing sequence.  
    }
    \label{fig:normalized}
\end{figure}

\begin{figure}[ht]
    \centering
    \includegraphics[scale=0.45]{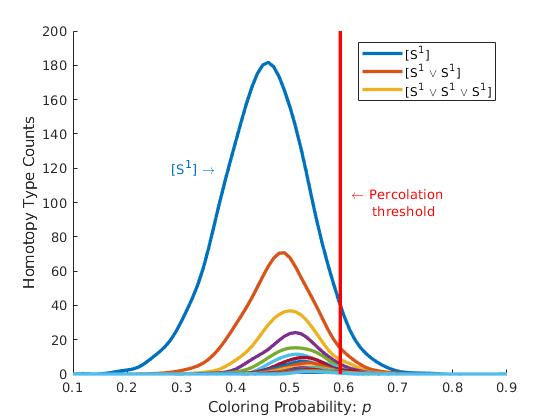}
    \caption{The homotopy type counts plotted (as a function of $p$) for each non-contractible homotopy type. The contractible homotopy type for this figure is omitted only because the number of contractible components dwarfs non-contractible components. Each homotopy type count appears to be unimodal with a maximal count attained at a value $p$ to the left of the percolation threshold $p_c$ and moving to the right as the homotopy type increases.  See Figure \ref{fig:normalized} for normalized plots depicting this trend.
    }
    \label{fig:toscale}
\end{figure}

\begin{figure}[ht]
    \centering
    \includegraphics[scale=0.5]{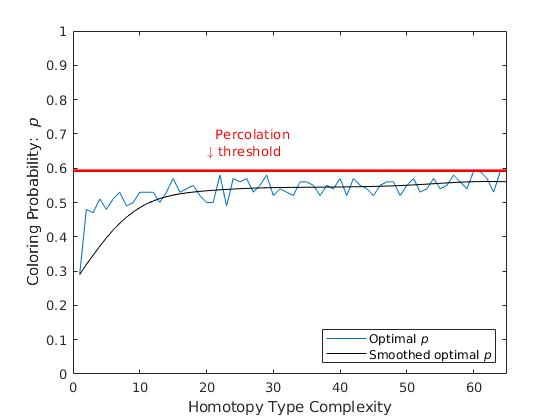}
    \caption{The probability $p$ at which the homotopy type count of wedges of $k$ copies of $S^1$ attains its maximum plotted against $k$, called homotopy type complexity. $k$ ranges from $0$ to $65$.
    We conjecture that this sequence of optimal probabilities is monotone and converges to the percolation threshold.
    }
    \label{fig:optimalp}
\end{figure} 

For our simulation, we chose $L = 250$ and tested probabilities from $0.1$ to $0.9$ with spacing $0.03$. Then, we averaged homotopy type counts over 20 iterations of the algorithm. Figures \ref{fig:normalized}, \ref{fig:toscale} plot these results. For figure $\ref{fig:optimalp}$, we chose a finer mesh with spacing $0.01$, but only took one iteration of the algorithm.





Observe that the probability where the $k$th homotopy type attains its maximal density appears below the percolation threshold $p_c$, and as $k$ gets larger, this sequence of ``optimal probabilities'' seems to increase.  We conjecture that this sequence converges monotonically to $p_c$. 
Our intuition is that the infinite component has countably infinite homotopy group generators, so as the number of homotopy group generators tends to infinity, we expect the optimal probability to approach the threshold. 

\section{Concluding remarks}\label{sec:concl}

\begin{itemize}
    \item[(i)] In light of \cite{LerarioMulas} it may be possible to generalize Theorem \ref{thm:main} to obtain a limit law for isotopy types while using an adaptation of the methods of Section \ref{sec:coeff}.
    It would also be interesting to study aspects of ambient isotopy of collections of components such as linking and nesting of components.
    \item[(ii)] It is natural to ask whether results analogous to Theorems \ref{thm:exptailmu}
    and \ref{thm:subexptailmu} hold
    in the setting of random geometric complexes \cite{ALL} or in the setting of random nodal domains \cite{SarnakWigman}.
   In order to make this question more precise, let us first consider a random geometric complex (built over a Poisson process in Euclidean space) and recall that the parameter analogous to $p$ is the radius $\alpha$ of the neighborhoods used to determine the simplices in the complex (see \cite{ALL}).
   The existence of a percolation threshold $\alpha=\alpha_c$ is known from the theory of continuuum percolation \cite{Penrose}, and we can ask whether the limiting homotopy measure
   described in \cite{ALL} has an exponential tail for $\alpha<\alpha_c$ and a subexponential tail for $\alpha>\alpha_c$.
   Concerning the setting of random nodal sets \cite{SarnakWigman},
   we note that Bogomolny and Schmit \cite{BogSch} have argued that sublevel domains  $\Omega_T:=\{x \in \RR^2 : \psi(x) < T \}$ of a Gaussian monochromatic random wave function $\psi$ can be regarded as a type of percolation process with
   the special case $T=0$ of nodal domains corresponding to critical percolation
   and the nonzero levels $T \neq 0$ corresponding to noncritical percolation.
   This suggests studying the
   distribution of homotopy types of the connected components of sublevel sets $\Omega_T$ while comparing 
   the tail decay rate for $T<0$ versus $T>0$.
    \item[(iii)]
    An open problem of particular interest is to study $\mu$ \emph{at the critical probability $p=p_c$}.  
    Percolation models at criticality have a diverging ``correlation length'', and the clusters exhibit structures more intricate than in either of the non-critical phases.
    In particular, for the random cubical complex at $p=p_c$ we conjecture a power-law tail decay of the homotopy measure $\mu$ (slower decay rate than either of the non-critical regimes addressed in Theorems \ref{thm:exptailmu} and \ref{thm:subexptailmu}).
    A pragmatic step might be to specialize this problem to the case of $d=2$ dimensions and instead consider critical percolation on the hexagonal grid (in light of progress surrounding Smirnov's breakthrough \cite{Smirnov}, this may even lead to explicit calculation of a ``critical exponent'' associated to homotopy tail decay).
    \item[(iv)] The above problem of studying tail decay of homotopy types in critical percolation fits into a more programmatic problem, to study topological features of critical phenomena (with percolation models at criticality being just one prototype).
    Some motivation for this problem is provided by the following strikingly topological description of water at ``critical opalescence'' (a critical state between liquid and gas) from \cite[Sec. 4.5]{Crowther}
    
    ``...within the bubbles of steam float liquid water droplets, and these droplets appear full of bubbles of steam, which, in turn, appear full of little droplets of liquid, and so on...''
    
    This description beckons us to study the topological complexity of critical phenomena
    and to look beyond homology and even homotopy to consider ambient \emph{isotopy types} of collections of components (the proper framework to capture nesting of components).
\end{itemize}

\bibliographystyle{abbrv}
\bibliography{cubical}




\end{document}